\documentclass{ijuc}
\usepackage[pdftex]{graphics}

\usepackage{amssymb,amscd}

\newtheorem{lemma}{Lemma}[section]
\newtheorem{theorem}[lemma]{Theorem}

\newtheorem{problem}[lemma]{Problem}

\newcommand{\comp}{{\rm Comp}}
\newcommand{\pol}{{\rm P}}
\newcommand{\mo}{{\rm mod}}
\newcommand{\Con}{{\rm Con}}
\newcommand{\Ker}{{\rm Ker}}
\newcommand{\GCD}{{\rm gcd}}
\newcommand{\bx}{\mathbf{x}}
\newcommand{\by}{\mathbf{y}}
\newcommand{\bo}{\mathbf{0}}

\newcommand{\Zpdva}{\mathbb{Z}_{p^2}}
\newcommand{\Zptri}{\mathbb{Z}_{p^3}}
\newcommand{\Zpnak}{\mathbb{Z}_{p^k}}
\newcommand{\Zpnakmin}{\mathbb{Z}_{p^{k-1}}}
\newcommand{\Zp}{\mathbb{Z}_p}
\newcommand{\Zm}{\mathbb{Z}_m}
\newcommand{\Zn}{\mathbb{Z}_n}
\newcommand{\Zmn}{\mathbb{Z}_m\times\mathbb{Z}_{n}}

\newcommand{\bu}{\mathbf{u}}

\newcommand{\bk}{\mathbf{k}}
\newcommand{\bc}{\mathbf{c}}
\newcommand{\ba}{\mathbf{a}}

\newcommand{\bl}{\mathbf{l}}

\newenvironment{proof}{P r o o f.}{\rule{2mm}{3mm}\vspace{5mm}}

\begin{document}
\title {Clones of Compatible Operations on Rings $\mathbb{Z}_{\MakeLowercase{p^k}}$}
\author{Miroslav Plo\v s\v cica\inst{1}\email{miroslav.ploscica@upjs.sk}
	\and Ivana Varga\inst{2} \email{ivana.varga@student.upjs.sk}
}

\institute{Institute of Mathematics, \v Saf\'arik's University, Ko\v sice, Slovakia
\and Institute of Mathematics, \v Saf\'arik's University, Ko\v sice, Slovakia}

\def\received{Received 24 June 2020; awaiting publication}

\maketitle
\begin{abstract}
We investigate the lattice $I(n)$ of clones on the ring $\Zn$ between the clone of polynomial functions
and the clone of congruence preserving functions. The crucial case is when $n$ is a prime power. 
For a prime $p$, the lattice $I(p)$ is trivial and $I(p^2)$ is known to be a $2$-element lattice. We provide a description
of $I(p^3)$. To achieve this result, we prove a reduction theorem, which says that $I(p^k)$ is isomorphic
to a certain interval in the lattice of clones on $\Zpnakmin$.
\end{abstract}
\keywords{congruence, clone, polynomial}

\section{Introduction and preliminaries}

A \emph{clone} on a set $A$ is a family of operations, which contains all projections and is closed under composition.
The family of all clones on $A$ forms a lattice.

An $n$-ary operation on an algebra $A$ is called \emph{compatible} or \emph{congruence preserving} if for $x_1,\dots, x_n,y_1\dots,y_n\in A$,
$(x_1,y_1),\dots,(x_n,y_n)\in\theta$ implies 
$(f(x_1,\dots,x_n),f(y_1,\dots,y_n))\in\theta$ for every 
$\theta\in\Con A$. It is clear that all compatible operations form a clone, denoted
by $\comp(A)$. This clone includes the clone $\pol(A)$ of all polynomial operations on $A$. Hence, $\pol(A)
\subseteq\comp(A)$. If the equality $\pol(A)=\comp(A)$ holds, then the algebra $A$ is called \emph{affine complete}.
A lot of research has been devoted to affine completeness for various types of algebras. Some survey can be find
in the monograph \cite{KP} of Kaarli and Pixley.

The algebras considered in this paper are the rings ${\mathbb Z}_n$ of integers modulo $n$. These rings are known to be
affine complete if and only if $n$ is squarefree (a product of distinct primes). This can be deduced from general
theorems about affine complete rings, see \cite{KP}. For $n$ that is not squarefree, we have $\pol(\Zn)\subsetneq\comp(\Zn)$, and a natural question is to describe the interval between $\pol(\Zn)$ and $\comp(\Zn)$. Clearly, this interval is a 
lattice, let us denote it by $I(n)$. The crucial case is when $n=p^k$, a power of a prime $p$, $k\ge 2$. (See Theorem
\ref{rozklad}.)

\begin{problem} Describe the lattice $I(p^k)$ for a prime $p$ and $k\ge 2$.\end{problem}

The answer to this problem seems not to be known for $k>2$. The case $k=2$ has been solved by Remizov
in \cite{RE}, who showed that $I(p^2)$ is a $2$-element lattice. Alternative proofs of this result can be found
in Bulatov \cite{B}, and also in our present paper. (See Theorem \ref{zpdva}.) For the case $k>2$, only
partial results are available, see \cite{RE}, \cite{GA1}, \cite{GA2}, and \cite{ME1}. Especially, it is known that
the lattice $I(p^k)$ for $k>2$ is infinite.

In the present paper we prove a reduction theorem, which says that the lattice $I(p^k)$ is isomorphic to
the interval between the clones $E_2(\Zpnakmin)$  and $\comp(\Zpnakmin)$, where $E_2(\Zpnakmin)$ is the clone on $\Zpnakmin$
generated by addition, constants and the binary operation $g(x,y)=pxy$. Notice that the clone $E_2(\Zpnakmin)$ is smaller
than $\pol(\Zpnakmin)$, but includes all polynomials of the \emph{additive group} $\Zpnakmin$. So, the
description of $I(p^k)$ depends on a description of clones on $\Zpnakmin$ containing all group polynomials.
For $k=3$ such clones have been completely described by Bulatov in \cite{B}. (See also
Meshchaninov's paper \cite{ME2}.) Relying on this paper we are able
to provide a complete description of the lattice $I(p^3)$.

Research in the clone theory connected with the modular arithmetics has a long and rich history.
 Clones of polynomials on $\Zn$ were studied in papers \cite{SA}, \cite{SZ1}, \cite{SZ2}, 
\cite{BU1}, \cite{BU2} of Salomaa, Szendrei, Bulatov, and others. From a more recent research relevant
to our topic, we would like to mention the results of Idziak (\cite{ID}),  Aichinger and Mayr
(\cite{AM}), and Mayr (\cite{MA}). These papers investigate clones containing a group operation on 
a given finite set.

The clone of all compatible functions has been studied also for other kinds of algebras. For instance, the paper
\cite{PH} describes generators of this clone for distributive lattices.

Elements of $\Zn$ will be denoted $0$, $1$, $\dots$, $n-1$. Congruences on the ring $\Zn$ are the usual
congruences modulo $d$ for every $d|n$. For vectors in $\Zn^k$ we adopt the convention that
$\bx=(x_1,\dots,x_k)$, $\bl=(l_1,\dots,l_k)$, etc.

\section{The ring $\mathbb{Z}_{\MakeLowercase{mn}}$}

Let $m,n$ be integers such that $\GCD(m,n)=1$. It is well known that the ring $\mathbb{Z}_{mn}$ is isomorphic to the product of rings $\Zmn$. Congruences of the ring $\Zmn$ are $\theta_{q,r}$ such that
$$((x,y),(u,v))\in\theta_{q,r}\ \hbox{iff}\ x\equiv u(\mo \ q)\ \hbox{and}\ y\equiv v(\mo\ r), $$ 
where $q$ is a divisor of $m$ and $r$ is a divisor of $n$. Let $I_{mn}$ be the interval of clones between 
$\pol(\Zmn)$ and $\comp(\Zmn)$. Clearly, the lattice $I_{mn}$ is isomorphic to $I(mn)$  and we show that
it is  isomorphic to $I(m)\times I(n)$.

\begin{lemma}\label{Zmn}
	Let $m,n$ be coprime numbers. Compatible operations on the ring $\Zmn$ are precisely operations of the form $$f((x_1,y_1),\dots,(x_l,y_l))=(f^{(m)}(\bx),f^{(n)}(\by)),$$
	where $f^{(m)}$ is compatible on $\Zm$ and $f^{(n)}$ is compatible on $\Zn$.
\end{lemma}
\begin{proof}
	Let $f$ be an $l$-ary compatible operation on $\Zmn$. Let $p_1$ denote the projection from $\Zmn$ to $\Zm$. Then $\theta_1=\Ker(p_1)$ is a congruence on $\Zmn$. Let $p_2$ denote the projection to $\Zn$, then $\theta_2=\Ker(p_2)$ is a congruence on $\Zmn$. 
	Since $f$ preserves the congruences $\theta_1$ and $\theta_2$,we have
	$$ f((x_1,y_1),\dots,(x_l,y_l))\theta_1 f((x_1,0),\dots,(x_l,0));$$
	$$ f((x_1,y_1),\dots,(x_l,y_l))\theta_2 f((0,y_1),\dots,(0,y_l)) $$
	for arbitrary $\bx\in(\Zm)^l, \by\in(\Zn)^l$. Set $$f^{(m)}(\bx)=p_1(f((x_1,0),\dots, (x_l,0)))$$  $$f^{(n)}(\by)=p_2(f((0,y_1),\dots, (0,y_l))).$$	
	Then $$p_1(f((x_1,y_1),\dots,(x_l,y_l)))=f^{(m)}(\bx)$$ $$p_2(f((x_1,y_1),\dots,(x_l,y_l)))=f^{(n)}(\by).$$
	We proved that there exist operations $f^{(m)},f^{(n)}$, such that $f=(f^{(m)},f^{(n)})$.
	
	It is clear that $f$ preserves $\theta_{q,r}$ if and only if $f^{(m)}$ preserves the congruence modulo $q$ and 
$f^{(n)}$ preserves the congruence modulo $r$. So, $f$ is compatible if and only if $f^{(m)},f^{(n)}$ are compatible. 
\end{proof}

It is easy to see the following.

\begin{lemma}\label{skladanie}
	 For every $l$-ary operation $f$ and all $k$-ary operations $g_1,\dots,g_l$  from a clone $C\in I_{mn}$, the following holds:
	$$(f(g_1,\dots,g_l))^{(m)}=f^{(m)}(g_1^{(m)},\dots,g_l^{(m)}).$$
\end{lemma}

\begin{lemma}\label{bijekcia}
	Let $m,n$ be coprime numbers.\begin{itemize}
		\item[(i)] Let $C\in I_{mn}$. Then the set $C_m=\lbrace f^{(m)}\mid f\in C\rbrace\in I(m)$
and the set $C_n=\lbrace f^{(n)}\mid f\in C\rbrace\in I(n)$.
		\item[(ii)] Let $D\in I(m)$ and $E\in I(n)$. Then 
$$D\times E=\lbrace f\mid f^{(m)}\in D, f^{(n)}\in E\rbrace\in I_{mn}.$$
		\item[(iii)] For every $C\in I_{mn}$: $C=C_m\times C_n$.
		\item[(iv)] For every $D\in I(m)$, $E\in I(n)$: $(D\times E)_m=D, (D\times E)_n=E$.
	\end{itemize}
\end{lemma}

\begin{proof}
	(i) It follows from Lemma \ref{Zmn} that the set $C_m$ is a subset of  $\comp(\Zm)$. It is closed under composition (from Lemma \ref{skladanie}) and  contains all projections, addition, multiplication and constants, because $C$ contains projections, addition, multiplication and constants. Analogously we can show that $C_n\in I(n)$.
	
	(ii)  It follows from Lemma \ref{Zmn} that the set $D\times E$ is a subset of the clone $\comp(\Zmn)$. It is closed under composition (from Lemma \ref{skladanie}) and contains all projections. Let $g^{(m)}(x_1,x_2)=x_1+x_2$ and $g^{(n)}(y_1,y_2)=y_1+y_2$. Then we have $g((x_1,y_1),(x_2,y_2))= (g^{(m)}(x_1,x_2),g^{(n)}(y_1,y_2))=(x_1+x_2,y_1+y_2)$, therefore $D\times E$ contains addition. Similarly, $D\times E$ contains multiplication and constants. 
	
	(iii) We want to prove that $C=\lbrace f\mid f^{(m)}\in C_m, f^{(n)}\in C_n\rbrace$. If $f\in C$, then 
 $f^{(m)}\in C_m$ and  $f^{(n)}\in C_n$ directly from the definition. Conversely, let $f^{(m)}\in C_m$ and $f^{(n)}\in C_n$. That means there exist functions $g,h\in C$, such that 
	$g^{(m)}=f^{(m)}$ and $h^{(n)}=f^{(n)}$. Using the Chinese remainder theorem, there exist natural 
numbers $a$, $b$ such that $a\equiv 1 (\mo\ m)$, $a\equiv 0(\mo\ n)$, $b\equiv 1 (\mo\ n)$, and $b\equiv 0(\mo\ m)$. Put $$f_1=\underbrace{g+g+\dots+g}_{\hbox{a-times}}+\underbrace{h+h+\dots+h}_{\hbox{b-times}}$$ 
Since $C$ contains the addition, we have $f_1\in C$. Moreover, 
$$f_1^{(m)}(\bx)=g^{(m)}(\bx)\cdot a+
h^{(m)}(\bx)\cdot b=g^{(m)}(\bx)=f^{(m)}(\bx)$$
 and 
$$f_1^{(n)}(\bx)=g^{(n)}(\bx)\cdot a+h^{(n)}(\bx)\cdot b=h^{(n)}(\bx)=f^{(n)}(\bx),$$
hence $f=f_1\in C$.

	(iv) Let $f \in D\times E$, then $f^{(m)}\in D$ from the definition. Let $g\in D$, then there exists an operation $f\in D\times E$, such that $f^{(m)}=g$. So, $g\in (D\times E)_m$, which proves the statement. Analogously we can show that $ (D\times E)_n=E$.
\end{proof}

Lemma \ref{bijekcia} shows that the assignments $C\mapsto (C_m, C_n)$ and $(D,E)\mapsto D\times E$
are mutually inverse bijections between $I_{mn}$ and $I(m)\times I(n)$. Clearly, they are order preserving,
so the lattices $I_{mn}$ and $I(m)\times I(n)$ are isomorphic. Using induction, we obtain the following result.
	
	\begin{theorem} Let $n=p_1^{\alpha_1}\cdot\dots\cdot p_m^{\alpha_m}$, where $p_1,\dots, p_m$ are
distinct primes. Then the lattice $I(n)$ is isomorphic to the product of the lattices $I(p_i^{\alpha_i})$,
$i=1,\dots,m$.\label{rozklad}
	\end{theorem}

The above theorem can be also deduced from \cite{RE}, in a different formalism. We include the proof here
for the sake of completeness.

So, in order to describe $I(n)$ it suffices to investigate the lattices $I(p^k)$ for a prime $p$ and $k\ge 2$.

\section{ Reduction of $\mathbb{Z}_{\MakeLowercase{p^k}}$ to $\mathbb{Z}_{\MakeLowercase{p^{k-1}}}$}

Let $k\ge 2$ be a fixed integer. Let $M=\lbrace lp \mid l\in\lbrace0,\dots, p^{k-1}-1\rbrace\rbrace$ denote the set of multiples of $p$ in $\Zpnak$. 

\begin{lemma} There exists a polynomial function $G$ on the ring $\Zpnak$ satisfying
 $$G(\bx)=\left\{ \begin{array}{l@{\quad}c}
	1, \hbox{ if }\ \bx \in M^n\\
	0, \hbox{ otherwise. }\\ \end{array} \right.$$ \label{funkciag}
\end{lemma}

\begin{proof} For every $x\in M$, we have $\displaystyle \prod_{c\in\Zpnak\setminus M}(x-c)=\displaystyle \prod_{c\in\Zpnak\setminus M}c=\alpha$, where the latter is invertible and independent of $x$. 
It is not difficult to prove that $\alpha$ is equal to $(-1)$ for $p>2$ and $1$ for $p=2$. 
Then the polynomial
$$G(\bx)=\alpha^{-n}{\displaystyle \prod_{i=1}^{n} \displaystyle \prod_{c\in\Zpnak\setminus M}(x_i-c)}$$ 
has the required properties. 
\end{proof}

Let $L=\{0,1,\dots,p-1\}\subseteq\Zpnak$. For every 
$\bx\in\Zpnak^n$ there exists a unique $\bc\in L^n$
such that $\bx-\bc\in M^n$. Thus,
Lemma \ref{funkciag} enables the following easy decomposition.

\begin{lemma}\label{decomp}
For every $n$-ary function $f$ on $\Zpnak$, 
    $$f(\bx)=\sum_{\bc\in L^n}f(\bx)G(\bx-\bc).$$
\end{lemma}

 Let $\pol(M)$ denote the clone on $M$ generated by addition and multiplication modulo $p^k$ and constant operations. Let $\comp(M)$ denote the clone that consists of all operations that preserve congruences modulo $p^2,\dots, p^{k-1}$. It is obvious that $\pol(M)\subseteq \comp(M)$. 

Now we are going to prove that the interval in the lattice of clones between $\pol(M)$ and $\comp(M)$
is isomorphic to $I(p^k)$. 

We say that $f$ preserves $M$, if $f(\bx)\in M$ whenever $\bx\in M^n$ .
For any clone $C\in I(p^k)$ (that is, for every clone between $\pol(\Zpnak)$ and $\comp(\Zpnak)$) we define  
$$C_M=\lbrace f \restriction M \mid f\in C, f \hbox{ preserves } M\rbrace.$$
 We show that assignment $C \mapsto C_M$ is the required isomorphism.
	
\begin{lemma}
For every clone $C\in I(p^k)$, the set $C_M$ is a clone between $\pol(M)$ and $\comp(M)$.
\end{lemma}

\begin{proof}
It is clear that $C_M$ is a clone. The rest follows from the fact, that the restriction of every $f\in\comp(\Zpnak)$
belongs to $\comp(M)$, and the addition and multiplication on $\Zpnak$ restrict to the addition and multiplication
on $M$.
\end{proof}

Conversely, for every clone $K$ on $M$ between $\pol(M)$ and $\comp(M)$ we define 
$$C(K)=\lbrace  f\in \comp(\Zpnak)\mid  \forall \ba\in \Zpnak^n : (f(\bx+\ba)-f(\ba))\restriction M\in K\rbrace.$$
Notice that the compatible operation  $f$ preserves the congruence $\mo\ p$, which implies that the operation
$f(\bx+\ba)-f(\ba)$ preserves $M$. We obtain that the restriction $(f(\bx+\ba)-f(\ba))\restriction M$ is in $\comp(M)$.

\begin{lemma}
For every clone $K$ between $\pol(M)$ and $\comp(M)$, the set  $C(K)$ is a clone in $I(p^k)$.
\end{lemma}
\begin{proof}
1. For every $n$-ary projection $f(\bx)=x_i$ on $\Zpnak$ and every $\ba\in \Zpnak^n$, the operation
$f(\bx+\ba)-f(\ba)=x_i+a_i-a_i=x_i$ is a projection on $M$ and hence belongs to $K$. Therefore,
$C(K)$ contains all projections.

2. To show that $C(K)$ is closed under composition, consider operations $f$, $g_1,\dots,g_n\in C(K)$, with $f$ $n$-ary and
all  $g_i$  $m$-ary.  Let $\ba\in\Zpnak^m$. Then $(g_1(\ba),\dots,g_n(\ba))\in\Zpnak^n$. Since $f\in C(K)$,
the operation
$$h(\bx)=(f(\bx+(g_1(\ba),\dots,g_n(\ba))-f(g_1(\ba),\dots,g_n(\ba)))\restriction M$$
belongs to $K$. Further, for every $i$, the operation 
$$h_i(\bx)=(g_i(\bx+\ba)-g_i(\ba))\restriction M$$
belongs to $K$ because $g_i\in C(K)$. Since $K$ is closed under composition, we have $h(h_1,\dots,h_n)(\bx)\in K$. For every $\bx\in M^m$ we have  $$h(h_1,\dots,h_n)(\bx)=h(h_1(\bx),\dots,h_n(\bx))=$$ $$h(g_1(\bx+\ba)-g_1(\ba),\dots,g_n(\bx+\ba)-g_n(\ba))=$$$$ f(g_1(\bx+\ba),\dots,g_n(\bx+\ba))-f(g_1(\ba),\dots,g_n(\ba))$$ Therefore $f(g_1,\dots,g_n)(\bx)$ belongs to $C(K)$.

3. It is easy to check that $C(K)$ contains the addition.

4. Let $f(x,y)=x\cdot y$, $a_1,a_2\in\Zpnak$. Then $f(x+a_1,y+a_2)-f(a_1,a_2)=(x+ a_1)\cdot(y+ a_2)-(a_1\cdot a_2)=x\cdot y+x\cdot a_2+y\cdot a_1$. The restriction of this function to $M$ belongs to $K$, because it
is a polynomial. (Notice that $x\cdot a_1$ can be replaced by the sum of $a_1$ copies of $x$, and similarly for $y\cdot a_2$.) Therefore $f\in C(K)$.

5. Every constant operation $f(\bx)=c$ belongs to the clone $C(K)$ because $f(\bx+\ba)-f(\ba)=0$, and the restriction
of this function belongs to $K$.

6. The inclusion $C(K)\subseteq\comp(\Zpnak)$ follows directly from the definition.
\end{proof}

\begin{lemma} 	$K=C(K)_M$ for every clone $K$ between the clones $\pol(M)$ and $\comp(M)$.\label{surj}
\end{lemma}

\begin{proof}
Let $h\in K$ be an $n$-ary operation on $M$. Define 
$$f(\bx)=\left\{ \begin{array}{l@{\quad}c}
h(\bx), \hbox{ if }\ \bx \in M^n\\
0, \hbox{ otherwise. }\\ \end{array} \right.$$ 
Then $f$ preserves $M$ and $h=f\restriction M$. We claim that $f\in C(K)$. The operation $f$ is compatible because $h$ is compatible. We need to show that $(f(\bx+\ba)-f(\ba))\restriction M\in K$ holds for every $\ba\in \Zpnak^n$ . For $\ba\in M^n$ we have $(f(\bx+\ba)-f(\ba))\restriction M=h(\bx+\ba)-h(\ba)\in K$ and for $\ba\notin M^n$ we have $(f(\bx+\ba)-f(\ba))\restriction M=0\in K$. This proves that $K\subseteq C(K)_M$.

Conversely, let $h\in C(K)_M$, which means that there exists $f\in C(K)$, such that $h=f\restriction M$ and $f$ preserves $M$. Using the definition of $C(K)$ with  $\ba=\bo $ we obtain that $(f(\bx+\bo)-f(\bo))\restriction M\in K$.
Since $f(\bo)\in M$, and $K$ contains constants and the addition, we have $h(\bx)=f(\bx)\restriction M\in K$. 
This proves that  $C(K)_M\subseteq K$.
\end{proof}

\begin{lemma}\label{order}
Let $C,D\in I(p^k)$. Then  $C\subseteq D$ if and only if $C_M\subseteq D_M$.
\end{lemma}

\begin{proof}
If $C\subseteq D$ then it is obvious that also  $C_M\subseteq D_M$. Conversely, suppose that $C_M\subseteq D_M$. Let $f\in C$ be an $n$-ary operation. Then, for every $\bc\in \{0,\dots,p-1\}^n$,
the operation $f(\bx+\bc)$ belongs to $C$. The $n$-ary constant function  $f(\bc)$ is also in $C$ and 
therefore  $f(\bx+\bc)-f(\bc)\in C$. The compatibility of $f$ implies that this function preserves $M$,
so the operation $(f(\bx+\bc)-f(\bc))\restriction M$
belongs to $C_M\subseteq D_M$. Consequently, there exists an operation $g_{\bc}\in D$ such that $g_{\bc}(\bx)=f(\bx+\bc)-f(\bc)$ for every $\bx\in M^n$. Using the polynomial $G$ from
Lemma \ref{funkciag}  we get 
$$g_{\bc}(\bx)G(\bx)=(f(\bx+\bc)-f(\bc))G(\bx),$$
which holds for every $\bc$ and every $\bx\in\Zpnak^n$. (It is trivial for $\bx\notin M^n$.)
After the substitution $\bx=\bu-\bc$ we obtain
$$g_{\bc}(\bu-\bc)G(\bu-\bc)=(f(\bu)-f(\bc))G(\bu-\bc)$$
for every $\bu\in\Zpnak^n$.
Now we use Lemma \ref{decomp}:
$$f(\bu)=\sum_{\bc}f(\bu)G(\bu-\bc)= \sum_{\bc}g_{\bc}(\bu-\bc)G(\bu-\bc)+\sum_{\bc}f(\bc)G(\bu-\bc)$$
and from this expression it follows that $f\in D$.
\end{proof}
	
\begin{theorem}\label{first}
The lattice $I(p^k)$  is isomorphic to the interval between $\pol(M)$ and $\comp(M)$.
\end{theorem}

\begin{proof} By Lemma \ref{order}, the assignment $C\mapsto C_M$ is an order embedding. By Lemma \ref{surj},
it is also surjective.
\end{proof}

So, the assignment $C\mapsto C_M$ is a bijection and hence has a unique inverse. According to Lemma \ref{surj},
this inverse is the assignment $K\mapsto C(K)$. Hence, we also have the following assertion.

\begin{lemma} 	$D=C(D_M)$ for every $D\in I(p^k)$.\label{invers}
\end{lemma}

\begin{lemma} 	Let $K$ be a clone on $M$ generated by addition, multiplication, constants and operations 
$\lbrace h_i\mid i\in I\rbrace$. Then $C(K)$ is generated by addition,\linebreak multiplication, constants and operations $f_i(\bx)= \left\{ \begin{array}{l@{\quad}c}
		h_i(\bx), \hbox{ if }\ \bx \in M^n\\
		0, \hbox{ otherwise. }\\ \end{array} \right.$\label{geners}
\end{lemma}	

\begin{proof}
	Let $D$ be generated by such operations on $\Zpnak$. These generators belong to $C(K)$, therefore $D\subseteq C(K)$. It is clear, that $D_M\supseteq K$ because $D_M$ contains all generators of $K$ and that yields $D=C(D_M)\supseteq C(K)$.
\end{proof}

As the second step in our reduction from $\Zpnak$ to $\Zpnakmin$ we now show that the interval between
$\pol(M)$ and $\comp(M)$ is isomorphic to a certain interval in the lattice of clones on $\Zpnakmin$.
The key is in the following construction.

Let $f:(\Zpnakmin)^n\rightarrow\Zpnakmin$ be an $n$-ary operation and define the operation $f^*:M^n\rightarrow M$ for every $l_i\in \Zpnakmin$ as follows
	$$f^*(l_1p,\dots,l_np)=f(l_1,\dots,l_n)\cdot p.$$ 
Recall that we identify $\Zpnakmin$ as the set $\{0,1,\dots,
p^{k-1}-1\}$. The numbers $l_i$ on the left hand side of the above equation are treated as elements of $\Zpnak$.	
This definition is correct, as every element of $M$ is equal to $lp$ for some $l\in\Zpnakmin$. It is easy to see
the following assertion.

\begin{lemma} An operation $f:(\Zpnakmin)^n\rightarrow\Zpnakmin$ belongs to $\comp(\Zpnakmin)$ if
and only if $f^*\in\comp(M)$.\label{compat}
\end{lemma}

\begin{lemma}\label{compo}
	Let $m,n\in \mathbb{N}$. For every $n$-ary operation $f$ and all $m$-ary operations $g_1,\dots,g_n$ on $\Zpnakmin$, 
the following holds:
	$$ (f(g_1,\dots,g_n))^*=f^*(g_1^*,\dots,g_n^*).$$
\end{lemma}
\begin{proof}
For every $\bl=(l_1,\dots,l_m)\in (\Zpnakmin)^m$ we compute
$$(f(g_1,\dots,g_n))^*(\bl p)=f(g_1,\dots,g_n)(\bl)\cdot p=f(g_1(\bl ),\dots,g_n(\bl))\cdot p,$$
and
$$f^*(g_1^*(\bl p),\dots,g_n^*(\bl p))=f^*(g_1(\bl)\cdot p,\dots,g_n(\bl)\cdot p)=f(g_1(\bl),\dots,g_n(\bl))\cdot p.$$
We have the required equality.
\end{proof}

	Let $E_2(\Zpnakmin)$ denote the clone on the ring $\Zpnakmin$ generated by constants, addition  and the binary operation $pxy$.
We show that the interval between $\pol(M)$ and $\comp(M)$ and the interval between $E_2(\Zpnakmin)$ and $\comp(\Zpnakmin)$ are isomorphic.
	
Let $C$ be a clone between $E_2(\Zpnakmin)$ and $\comp(\Zpnakmin)$  and let $C^*$ denote the set $$C^*=\lbrace f^* \mid f\in C\rbrace.$$

\begin{lemma}
	For every clone $C$ between $E_2(\Zpnakmin)$ and $\comp(\Zpnakmin)$, $C^*$ is a clone between $\pol(M)$ and $\comp(M)$.
\end{lemma}
\begin{proof}
1. If $f$ is a projection on $\Zpnakmin$, then $f^*$ is the same projection on $M$. By Lemma \ref{compo},
$C^*$ is closed under composition. Thus, $C^*$ is indeed a clone.

2. If $f$ is the addition on $\Zpnak$, then $f^*$ is the addition $\mo\ p^k$ on $M$. If $g(x,y)=pxy$ 
on $\Zpnakmin$, then $g^*(xp,yp)=g(x,y)\cdot p=(pxy)p=px\cdot py$ is the multiplication modulo $p^k$
on $M$. Further, every constant operation  $h(\bx)=c\in\Zpnakmin$ belongs to $C$, so $h^*(lp)=cp$ is a
constant operation on $M$. We have obtained that $C^*\supseteq\pol(M)$.

3. Every $f\in C$ preserves congruences modulo $p,\dots, p^{k-2}$. Consequently, then $f^*$ preserves 
congruences modulo $p^2,\dots,p^{k-1}$. The congruence 
$\mo\ p$ is trivial on $M$. Hence,
$C^*\subseteq\comp(M)$.
\end{proof}

Clearly, $f^*=g^*$ if and only if $f=g$. Hence, the assignment
$C\mapsto C^*$ is an order embedding. Now we show its 
surjectivity.
	
\begin{lemma}
For any clone $D$ between $\pol(M)$ and $\comp(M)$, the set 
$$C=\lbrace h\in\comp(\Zpnakmin)\mid h^*\in D\rbrace$$
 is a clone between $E_2(\Zpnakmin)$ and $\comp(\Zpnakmin)$. Moreover, $D=C^*$.
\end{lemma}

\begin{proof}
It is clear that $C$ is a clone. (The closedness under composition follows from Lemma \ref{compo}.) Moreover, 
if $f$ is the addition on $\Zpnakmin$, then $f^*$ is the addition on $M$, which belongs to $D$, so $f\in C$. 
Similarly, if $g(x,y)=pxy$ on $\Zpnakmin$, then $g^*$ is the multiplication on $M$, so $g^*\in D$ and hence
$g\in C$. If $h$ is a constant operation on $\Zpnakmin$, then $h^*$ is a constant operation on $M$ and, again,
$h^*\in D$ implies $h\in C$. We have proved that $E_2(\Zpnakmin)\subseteq C$. By Lemma \ref{compat} we have 
$C\subseteq\comp(\Zpnakmin)$. 

It remains to prove that $D=C^*$. The inclusion $C^*\subseteq D$ is trivial. Conversely, let $f\in D$. Then
there is an operation  $h$ on $\Zpnakmin$ such that $f(\bx p)=h(\bx)\cdot p$. Clearly, $f=h^*$. Lemma \ref{compat} implies that $h\in\comp(\Zpnakmin)$, so $h\in C$ and $f\in C^*$.
\end{proof}

As a consequence of previous lemmas we state the following theorem. 
\begin{theorem}\label{main}
The lattice $I(p^k)$ is isomorphic to the interval
between clones $E_2(\Zpnakmin)$ and $\comp(\Zpnakmin)$.
\end{theorem}

The isomorphism in our Theorem maps a clone $K$ between $E_2(\Zpnakmin)$ and $\comp(\Zpnakmin)$ first into
the clone $K^*$ and then, by Theorem \ref{first}, into $C(K^*)$. We also have a correspondence
between generators. If $\{f_i\mid i\in I\}$ is a generating set of $K$, then (by Lemma \ref{compo})
$\{f_i^*\mid i\in I\}$ is a generating set for $K^*$. The generating set of $C(K^*)$ is then described by
Lemma \ref{geners}.

\section{Cases $\MakeLowercase{k}=2$ and $\MakeLowercase{k}=3$}

The interval between $E_2(\Zpnakmin)$ and $\comp(\Zpnakmin)$ is known for $k=2$ and $k=3$. We can use this knowledge
to describe all clones between $P(\Zpnak)$ and $\comp(\Zpnak)$.

If $k=2$, then the operation $g(x,y)=pxy$ on ${\mathbb Z}_p$ is trivially zero. So, $E_2(\Zp)$ is the clone of all 
polynomials of the group $(\Zp,+)$. It is well known that this clone is maximal, which means that it is covered
by the clone of all operations on the set ${\mathbb Z}_p$. (It can be deduced from the well known 
Rosenberg's classification 
in \cite{RO}.) The clone of all operations coincides with $\comp(\Zp)$, since the ring $\Zp$ has only trivial
congruences and therefore all operations are compatible. We obtain the following result.

\begin{theorem}\label{zpdva}
The interval between $\pol(\Zpdva)$ and $\comp(\Zpdva)$ has only two elements. 
\end{theorem}

The case $k=3$ is much more complicated. The interval between 
 $E_2$ and $\comp(\Zpdva)$ is only known from Bulatov's paper \cite{B}. (We write $E_2$ instead of 
 $E_2(\Zpdva)$.) In fact, Bulatov described all clones, which contain polynomials of the group $(\Zpdva,+)$. The lattice of these clones is depicted below. Each clone is
determined by a set of generators, which always contains the addition and the constants. Notice that the picture
includes the fact stated in Theorem \ref{zpdva}.
\begin{theorem} The lattice of clones between $P(\Zptri)$ and $\comp(\Zptri)$ is isomorphic to the interval
$E_2$ and $\comp(\Zpdva)$ on the picture below.
\end{theorem}
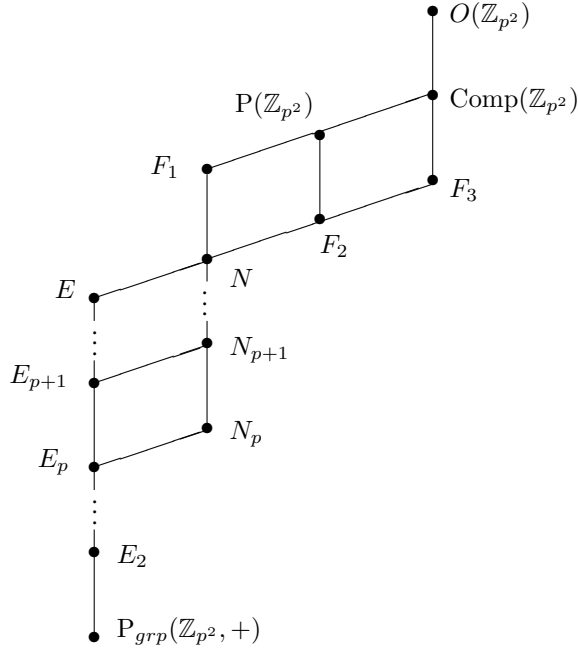
\unitlength=0,75mm
\begin{picture}(120,120)

\put(45,0){\circle*{2}}
\put(45,15){\circle*{2}}
\put(45,45){\circle*{2}}
\put(45,30){\circle*{2}}
\put(45,60){\circle*{2}}
\put(65,37){\circle*{2}}
\put(65,52){\circle*{2}}
\put(65,67){\circle*{2}}
\put(85,74){\circle*{2}}
\put(105,81){\circle*{2}}
\put(65,83){\circle*{2}}
\put(85,89){\circle*{2}}
\put(105,96){\circle*{2}}
\put(105,111){\circle*{2}}

\put(45,0){\line(0,1){19}}
\put(44.5,20){$\vdots$}
\put(45,26){\line(0,1){4}}
\put(45,30){\line(0,1){19}}
\put(44.5,50){\vdots}
\put(45,56){\line(0,1){4}}
\put(45,30){\line(3,1){20}}
\put(45,45){\line(3,1){20}}
\put(45,60){\line(3,1){20}}
\put(65,37){\line(0,1){19}}
\put(64,57){\vdots}
\put(65,63){\line(0,1){19}}
\put(65,67){\line(3,1){40}}
\put(65,83){\line(3,1){40}}
\put(85,74){\line(0,1){15}}
\put(105,81){\line(0,1){30}}

\put(49,0){$\pol_{grp}(\Zpdva,+)$}
\put(49,13){$E_2$}
\put(35,30){$E_p$}
\put(30,45){$E_{p+1}$}
\put(38,60){$E$}
\put(69,35){$N_p$}
\put(69,50){$N_{p+1}$}
\put(69,62){$N$}
\put(55,82){$F_1$}
\put(70,93){$\pol(\Zpdva)$}
\put(85,68){$F_2$}
\put(108,78){$F_3$}
\put(108,94){$\comp(\Zpdva)$}
\put(108,109){$O(\Zpdva)$}
\end{picture}
\vskip5mm

Now we list the generators of all clones $K$ between $E_2$ and $\comp(\Zpdva)$
(taken from \cite{B}), as well as the generators of the corresponding clones\linebreak $\Phi(K)=C(K^*)$ between
$P(\Zptri)$ and $\comp(\Zptri)$. The generators of $\Phi(K)$ are constructed by the process described at the 
end of the previous section.

The definitions are as follows. The operation $h$  on $\Zpdva$ is defined by the formula
$$ h(x,y)= \left\{ \begin{array}{l@{\quad}c}
klp, \hbox{ if }  x=kp, y=lp\ \hbox{for some}\ k,l\in\{0,\dots,p-1\}\\
0, \hbox{ otherwise. }\\ \end{array} \right.$$
The $j$-ary operation $\xi_j$ on $\Zptri$ is defined by
$$ \xi_j(\bx)= \left\{ \begin{array}{l@{\quad}c}
k_1\dots k_jp^2, \hbox{ if }  \bx=\bk p\ \hbox{for some}\ \bk\in\{0,\dots,p^2-1\}^j\\
0, \hbox{ otherwise. }\\ \end{array} \right.$$
Notice that $\xi_2$ is the restriction of the usual multiplication to $M$. It is a polynomial of the ring $\Zptri$,
so $\Phi(E_2)=\pol(\Zptri)$.

Next we define operations $\pi$, $\psi$, $\rho$, $\varphi$ and $\tau$ on $\Zptri$.

$$ \pi(x)= \left\{ \begin{array}{l@{\quad}c}
pk^p, \hbox{ if }  x=k p\ \hbox{for}\ k\in\{0,\dots,p^2-1\}\\
0, \hbox{ otherwise. }\\ \end{array} \right.$$

$$ \psi(x,y)= \left\{ \begin{array}{l@{\quad}c}
pk^pl^p, \hbox{ if }  x=k p, y=lp\ \hbox{for}\ k,l\in\{0,\dots,p^2-1\}\\
0, \hbox{ otherwise. }\\ \end{array} \right.$$

$$ \rho(x,y)= \left\{ \begin{array}{l@{\quad}c}
pk^p(l^p-l), \hbox{ if }  x=k p, y=lp\ \hbox{for}\ k,l\in\{0,\dots,p^2-1\}\\
0, \hbox{ otherwise. }\\ \end{array} \right.$$

$$ \varphi(x,y)= \left\{ \begin{array}{l@{\quad}c}
klp^2, \hbox{ if }  x=k p^2, y=lp^2\ \hbox{for}\ k,l\in\{0,\dots,p-1\}\\
0, \hbox{ otherwise. }\\ \end{array} \right.$$

$$ \tau(x,y)= \left\{ \begin{array}{l@{\quad}c}
klp, \hbox{ if }  x=k p, y=lp\ \hbox{for}\ k,l\in\{0,\dots,p^2-1\}\\
0, \hbox{ otherwise. }\\ \end{array} \right.$$

The generators of all clones are in the following table.
We only list the additional generators (besides addition and constants for $K$,
besides addition, multiplication and constants for $\Phi(K)$). The clones $E$ and $N$ are not in the table,
they are the union of all $E_j$ and $N_j$, respectively.
\vskip5mm

\begin{center}
\begin{tabular}{ |c|c|c| } 
\hline
 $K$ & generators of $K$ & generators of $\Phi(K)$\\
\hline
\hline
$E_j$ & $px_1\dots x_j$ & $\xi_j$\\ 
\hline
$N_j$ & $px_1\dots x_j$, $x^p$ & $\xi_j$, $\pi$\\ 
\hline
$F_1$ & $x^py^p$ & $\psi$\\
\hline
$F_2$ & $x^p(y^p-y)$ & $\rho$\\
\hline
$F_3$ & $h$ & $\varphi$\\
\hline
$\pol(\Zpdva)$ & $xy$ & $\tau$\\
\hline
$\comp(\Zpdva)$ & $xy$, $h$ & $\tau$, $\varphi$\\
\hline
\end{tabular}
\end{center}

\section{Acknowledgements}
This work has been supported by Slovak VEGA grant 1/0097/18.


\vskip5mm
\begin{thebibliography}{99}


\bibitem{AM} E. Aichinger, P. Mayr, \emph{Polynomial clones on groups of order $pq$},
Acta Math. Hungar. {\bf 114} (2007), 267-285.

\bibitem{BU1} A. A. Bulatov, \emph{Polynomial reducts of modules I. Rough classification},
Multiple-valued Logic {\bf 3} (1998), 135-154.

\bibitem{BU2} A. A. Bulatov, \emph{Polynomial reducts of modules II. Algebras of Primitive and nilpotent functions},
Multiple-valued Logic {\bf 3} (1998), 173-193.

\bibitem{B} A. A. Bulatov, \emph {Polynomial clones containing the Mal'tsev operation of the groups $\Zpdva$ and $\Zp\times\Zp$}, Multiple-valued Logic {\bf 8} (2002), 193-221.

\bibitem{GA1} G. P. Gavrilov, \emph{On the overstructure of the class of polynomials of multivalued logics} (in Russian),
Diskretnaja Matematika {\bf 8} (1996), 90-97.

\bibitem{GA2} G. P. Gavrilov, \emph{On closed classes of multi-valued logic that contain the class of polynomials} (in Russian), 
Diskretnaja Matematika {\bf 9}, (1997), 12-23.

\bibitem{ID} P. Idziak, \emph{Clones with Mal'tsev operation},
Internat. J. of Algebra and Computation {\bf 9} (1999), 213-226.

\bibitem{KP} K. Kaarli, A. F. Pixley, \emph{Polynomial completeness in algebraic systems},
Chapman \& Hall/ CRC, 2001. 

\bibitem{MA} P. Mayr, \emph{Polynomial clones on squarefree groups},
Internat. J. of Algebra and Computation {\bf 18} (2008), 759-777.

\bibitem{ME1} D. G. Meshchaninov, \emph{On some properties of the overstructure of classes of polynomials in 
$P_k$} (in Russian), Mat. Zametki {\bf 44} (1988), 673-681.

\bibitem{ME2} D. G. Meshchaninov, \emph{A family of Closed Classes in $k$-Valued Logic}, 
Moscow University Computational Mathematics and Cybernetics {\bf 43} (2019), 25-31.

\bibitem{PH} M. Plo\v s\v cica, M. Haviar, \emph{Congruence-preserving functions on distributive lattices},
Algebra Universalis {\bf 59} (2008), 179-196.

\bibitem{RE} A. B. Remizov, \emph{On the overstructure of closed classes of polynomials modulo k} (in Russian),
Diskretnaja Matematika {\bf 1} (1989), 3-15.

\bibitem{RO} I. G. Rosenberg, \emph{\"Uber die funktionalle Vollst\"andigkeit in den mehrwertigen Logiken},
Roz- pravy \v Ceskoslovensk\'e Akademie v\v ed, Ser. Math. Nat. Sci.  {\bf 80} (1970), 3-93.

\bibitem{SA} A. A. Salomaa, \emph{On infinitely generated sets of operations in finite algebras}, Ann. Univ. Turku, 
Ser. AI {\bf 74} (1964), 1-12.

\bibitem{SZ1}  \'A. Szendrei, \emph {Idempotent reducts of Abelian groups,}
Acta Sci. Math. (Szeged) {\bf 38} (1976), 171-182.

\bibitem{SZ2} \'A. Szendrei, \emph{Clones of linear operations on finite sets},
in: Finite algebra and multiple-valued logic, Colloq. Math. Soc. J. Bolyai {\bf 28} (1977), 693-738.


\end{thebibliography}
\end{document}